\documentclass[12pt]{amsart}
\usepackage{amsfonts}
\usepackage{amssymb}
\usepackage{color}
\usepackage{dsfont}
\numberwithin{equation}{section}
\theoremstyle{plain}
 \newtheorem{theorem}{Theorem}[section]
 \newtheorem{lemma}[theorem]{Lemma}
 \newtheorem{corollary}[theorem]{Corollary}
 \newtheorem{proposition}[theorem]{Proposition}

\theoremstyle{definition}
 
 \newtheorem{example}[theorem]{Example}
 \newtheorem{remark}[theorem]{Remark}

\newcommand{\bN}{\mathbb{N}}
\newcommand{\bZ}{\mathbb{Z}}
\newcommand{\bR}{\mathbb{R}}

\newcommand{\bC}{\mathbb{C}}

\newcommand{\one}{\mathbf{1}}

\allowdisplaybreaks

\begin{document}

\vspace{5mm}
\begin{center}
{\bf
{\large
On the integral modulus of infinitely divisible distributions}}

\vspace{5mm}

David Berger\\
\end{center}

We derive some estimates for the integral modulus of continuity of probability densities of infinitely divisible distributions. The paper is splitted into two parts. The first part deals with general infinitely divisible distributions. The second part is mainly concerned with densities of random integrals with respect to a L\'{e}vy process. We will see major differences between compact and non-compact supports.
\section{Introduction}
The modulus of continuity $||f(z-\cdot)-f(\cdot)||_{L^p(\bR)}$ for $z\in\bR$ has a deep connection to Fourier series and also to the Fourier transform. The decaying rate of the Fourier transform (or weighted versions, see [\ref{Bray}] and cited articles) can be estimated by the modulus of continuity and vice versa, where all these estimates depend on $p\ge 1$. For $1\le p\le 2$ it is hard to obtain estimates for the modulus of continuity in terms of the Fourier transform, but especially the case $p=1$ is very interesting as $\int_{\bR} |f(x-z)-f(x)|\lambda(dx)\le C |z|$ for all $z$ is equivalent to the fact that $f$ is of bounded variation (see [\ref{Ambrosio}, Exercise 3.3, p. 208]). \\
In statistics it is also interesting to know if a probability density is of bounded variation if one wants to estimate the density, see [\ref{Datta}, Theorem 3]. Moreover, if one has a linear process $X=(X_t)_{t\in\bZ}$ with $X_t=\sum_{i=0}^\infty a_i Z_{t-i}$ with $a_i\in \bR$ and $(Z_i)_{i\in\bZ}$ are iid random variables with Lebesgue density $f$, the strong mixing rate of the process $X$ depends on the modulus $||f(z-\cdot)-f(\cdot)||_{L^1(\bR)}$, see [\ref{Gorodetskii}, Theorem and proof].\\
In this paper we are mostly interested in special classes of infinitely divisible distributions. The paper is separated into two parts. The first part is interested in  general infinitely divisible distributions. A probability measure $\mu$ on $\bR$ is infinitely divisible, if there exist constants $\gamma\in \bR$, $a\ge 0$ and a L\'{e}vy measure $\nu$ on $\bR$ (i.e. a measure $\nu$ satisfying $\nu(\{0\})=0$ and $\int_{\bR} \min\{1,x^2\} \nu(dx)<\infty$) such that the Fourier transform $\hat{\mu}$ satisfies
\begin{align}\label{eq1}
\hat{\mu}(z)=\exp\left(-\frac{1}{2}a z^2+i\gamma z +\int\limits_{\bR} (e^{ixz}-1-ixz\one_{(-1,1)}(x)\nu(dx)\right)
\end{align}
for every $z\in\bR$. It can be shown that the triplet $(a,\gamma,\nu)$ is unique, and that for every such triplet $(a,\gamma,\nu)$  the right-hand side of (\ref{eq1}) defines the Fourier transform of an infinitely divisible distribution, see [\ref{Sato}, Theorem 8.1, p. 37].\\
 The normal distribution is itself an infinitely divisible distribution with characteristic triplet $(a,\gamma,0)$. If $a>0$ it has of course a Lebesgue density with very nice properties so it is not very suprising that we find bounds for the modulus and as a consequence we obtain for the larger class of distributions with characteristic triplet $(a,\gamma,\nu)$ with $a>0$ similiar estimates for the integral modulus.\\
In the more complicated case $a=0$, we will give sufficient conditions on the characteristic triplet $(0,\gamma,\nu)$ to have H\"older bounds for the modulus $||f(z-\cdot)-f(\cdot)||_{L^1(\bR)}$ if $\nu(dx)$ has a Lebesgue density in a neighborhood of zero. \\
An important subclass of infinitely divisible distributions is the class of self-decomposable distributions. They are infinitely divisible distributions for which the L\'{e}vy measure has a density of the form $\frac{k(x)}{|x|}$, such that $k$ is increasing on $(-\infty,0)$ and decreasing on $(0,\infty)$, see [\ref{Sato}, Theorem 15.10, p. 95]. They have a Lebesgue density if they are non-degenerate. Furthermore, explicit bounds for the decay of their Fourier transform are known, so it seems natural to start the search for bounds with this class. An important property of these distributions is the unimodality. We will use this property in our proof for the main result. By using known estimates for the modulus and the decay of their Fourier transform it is possible to find upper bounds for the integral modulus and we will see that most of our results are in some sense optimal.\\
The second part of the paper deals with stochastic integrals of deterministic functions with respect to L\'{e}vy processes and their corresponding densities, where we consider compact and non-compact supports.  For the compact support we will deal with kernels which are $\mathcal{C}^1-$diffeomorphisms on their support. We will see that every stochastic integral with such a kernel has a Lebesgue density and derive necessary and sufficient conditions on the L\'{e}vy process and the kernel such that the density is of bounded variation. Based on this we will consider for the non-compact support $[0,\infty)$ kernels such that there exists a sequence $0=t_0<t_1<\dotso<t_n\to\infty$ for $n\to\infty$ such that the kernel is a $\mathcal{C}^1$-diffeomorphism in every $(t_i,t_{i+1})$ for every $i\in\bN_0$.  We will find sufficient conditions for the existence of a Lebesgue density of bounded variation and will especially see that there exist kernel functions such that the property of the existence of a BV-density is independent of the integrating L\'{e}vy-process, which is clearly not the case for the compact case.

\section{Notation and Preliminaries}
To fix notation, by a distribution on $\bR$ we mean a probability measure on $(\bR,\mathcal{B})$ with $\mathcal{B}$ being the Borel $\sigma-$algebra on $\bR$, and similarly, by a signed measure on $\bR$ we mean it to be defined on $(\bR,\mathcal{B})$. By a measure on $\bR$ we always mean a positive measure on $(\bR,\mathcal{B})$, i.e. a $[0,\infty]$-valued $\sigma-$additive set function on $\mathcal{B}$ that assigns the value $0$ to the empty set. The Dirac measure at a point $b\in \bR$ will be denoted by $\delta_{b}$, the Gaussian distribution with mean $a\in\bR$ and variance $b\ge 0$ by $N(a,b)$ and the Lebesgue measure by $\lambda(dx)$. The Fourier transform at $z\in\bR$ of a finite positive measure $\mu$ on $\bR$ will be denoted by $\hat{\mu}(z)=\int_\bR e^{ixz}\,\mu(dx)$. The convolution of two positive measures $\mu_1$ and $\mu_2$ on $\bR$ is defined by $\mu_1\ast \mu_2(B) =\int_\bR \mu_1(B-x)\,\mu_2(dx)$, $B\in\mathcal{B}$, where $B-x=\{y-x|\,y\in B\}$. The law of a random variable $X$ will be dentoted by $\mathcal{L}(X)$. The imaginary unit will be denoted by $i$. We write $\bN=\{1,2,\dotso\}$, $\bN_0=\bN\cup \{0\}$ and $\bZ,\,\bR,\,\bC$ for the set of integers, real numbers and complex numbers, respectively. The indicator function of a set $A\subset \bR$ is denoted by $\one_A$. By $L^1(\bR, A)$ for $A\subset \bC$ we denote the set of all Borel-measurable functions $f:\bR \to A$ such that $\int_\bR |f(x)|\,\lambda(dx)<\infty$. By $BV(\bR,\bR)$ we denote the set of functions $f:\bR\to\bR$ of bounded variation, which means for every decomposition $-\infty<a_1<\dotso<a_n<\infty$ it holds $\sum_{i=1}^{n-1} |f(a_i)-f(a_{i+1})|\le C<\infty$ for some $C>0$ independent of the decomposition. By $TV_f([a,b])$ we denote the total variation of the function $f\in BV(\bR,\bR)$ in the interval $[a,b]$.

\section{Densities of infinitely divisble distributions}
Our goal of this section is to prove some aspects of the integral modulus of continuty of densities from infinitely divisible distributions. We will specialize on infinitely divisible distributions with L\'{e}vy measure $\nu$ such that $|x|\nu(dx)$ has a Lebesgue density around a neighborhood of $0$.\\ 
As stated in the introduction the class of self-decomposable distributions is a subclass of such distributions. All self-decomposable distributions are unimodal, which will play a major rule in the proof of the main theorem, see [\ref{Sato}, Theorem 53.1, p. 404]. We will derive the main result by minorizing the L\'{e}vy measure by a L\'{e}vy measure corresponding to a self-decomposable distribution.\\
We start with an easy example and derive some bounds for the integral modulus of continuity of normal distributions and infinitely divisible distributions with a non-vanishing Gaussian variance.
\begin{lemma}\label{lem1}
Let $\mu_1$ be absolutely continuous with Lebesgue density $f$ and $\mu_2$ be a probability measure. Let $\int_\bR |f(x)-f(x-z)|\lambda(dx)\le h(z)$ for some  $z\in\bR$. Then for the Lebesgue density $g$ of $\mu_1\ast\mu_2$ holds also $\int_{\bR} |g(x)-g(x-z)|\lambda(dx)\le h(z)$. 
\end{lemma}
\begin{proof}
We know that $\mu_1\ast \mu_2$ is absolutely continuous with Lebesgue density $g(x)=\int_{\bR}f(x-y)\mu_2(dy)$. We see that
\begin{align*}
\int\limits_{\bR} |g(x)-g(x-z)|\lambda(dx)
\le& \int\limits_{\bR}\int\limits_{\bR}|f(x-y)-f(x-z-y)|\lambda(dx)\mu_2(dy)\\
\le & h(z)\int\limits_{\bR} \mu_2(dy)=h(z).
\end{align*}
\end{proof}
\begin{corollary}\label{cor10}
Let $\mu$ be an infinitely divisible distribution with characteristic triplet $(a,\gamma,\nu)$ such that $a>0$. Then $\mu$ is absolutely continuous and the Lebesgue density $f_\mu$ satifies
\begin{align*}
\int\limits_{\bR} |f_{\mu}(x)-f_{\mu}(x-z)|\lambda(dx)\le C|z|
\end{align*}
for some constant $C$ and every $z\in\bR$.
\end{corollary}
\begin{proof}
Let $\mu_1= N(0,a)$ be a normal distribution with mean $0$ and variance $a$. We have that $f_{\mu_1}(x)=1/\sqrt{2\pi a}\exp(-x^2/(2a))$ and find by a simple calculation that
\begin{align*}
\int\limits_{\bR} |f_{\mu_1}(x)-f_{\mu_1}(x-z)|\lambda(dx)=\sqrt{\frac{2}{\pi a}}\int\limits_{(-|z|/2,|z|/2)}\exp(-x^2/(2a))\lambda(dx),
\end{align*}
which is $O(|z|)$ for $|z|\to 0$. The rest follows by Lemma \ref{lem1}.
\end{proof}
\begin{remark}
We could have proven it in another way, as the density is continuous and bounded, but we wanted to show that it is not possible to obtain a better bound for normal distributions. I.e. the proof shows actually 
\begin{align*}
\lim_{z\to 0}|z|^{-1} \int\limits_{\bR} |f_{\mu}(x)-f_{\mu}(x-z)|\lambda(dx)=\sqrt{\frac{2}{\pi a}}
\end{align*}
when $\mu\sim N(0,a)$.
\end{remark}
Now we will state our main result and prove it directly. There are many consequences of this result and we will later show some applications to obtain further infinitely divisble distributions with a density of bounded variation.
\begin{theorem}\label{theorem1}
Let $\mu$ be an infinitely divisible distribution with characteristic triplet $(a,\gamma,\nu)$  where $a\ge 0$, $\gamma\in\bR$ and $\nu$ a L\'{e}vy measure such that $|x|\nu(dx)$ has a Lebesgue density $k$ in a neighborhood around zero with $\liminf_{x\to 0+}k(x)+\liminf_{x\to 0-}k(x)=:c_{\inf}$.
\begin{itemize}
\item[i)] If $c_{\inf}>1/p$ for $1<p\le 2$, then $\mu$ has a Lebesgue density \\$f_{\mu}\in L^1(\bR,\bR^+)\cap L^{p/(p-1)}(\bR,\bR^+)$ and there exists a constant $C>0$ such that
\begin{align*}
\int\limits_{\bR}|f_{\mu}(x-z)-f_{\mu}(x)|\,\lambda(dx)\le C|z|^{\frac{1}{p}}
\end{align*}
for every $z\in\bR$.
\item[ii)] If $c_{\inf}>1$, then $f$ is continuous on $\bR$ and there exists a constant $C>0$ such that
\begin{align*}
\int\limits_{\bR}|f_{\mu}(x-z)-f_{\mu}(x)|\,\lambda(dx)\le C |z|
\end{align*}
for every $z\in\bR$.
\item[iii)] Now let $c_{\sup}:=\limsup_{x\to 0+}k(x)+\limsup_{x\to 0-}k(x)<\frac{1}{p}$  with $p\in (0,\infty)$ and let $a=0$. Then, if $\mu$ has a Lebesgue density $f_{\mu}$, it satisfies
\begin{align}\label{eq3}
\sup_{0\le h \le |z|}\int\limits_{\bR}|f_{\mu}(x-h)-f_{\mu}(x)|\,\lambda(dx)\ge C |z|^{\frac{1}{p}}
\end{align}
for some constant $C>0$ and $z\in(-1,1)$.
\end{itemize}
\end{theorem}
\begin{proof}
For the proof assume that $a=0$ as otherwise the assertion would be implied by Corollary \ref{cor10}. For the proof of i) and ii) we assume first that $k$  is increasing on $(-\delta,0)$ and decreasing on $(0,\delta)$ for some $\delta>0$ and else $0$ such that $(0,\gamma,\frac{k(x)}{|x|}\lambda(dx))$ is the characteristic triplet of a self-decomposable distribution $\mu$, see [\ref{Sato}, Theorem 15.10].\\
i) We then know that $c=c_{\inf}=c_{\sup}=k(0+)+k(0-)>0$ . Then it holds true that $|\hat{\mu}(z)|=o(|z|^{-\alpha})$ as $|z|\to\infty$ with $0<\alpha <c$, see [\ref{Sato}, Lemma 28.5, p. 191]. If $c>\frac{1}{p}$, it follows that $\hat{\mu}\in L^p(\bR,\bC)$ and we conclude that $f_{\mu}\in L^{p^*}(\bR,[0,\infty))$, see [\ref{Grafakos}, Proposition 2.2.16., p. 104], where $p^*=\frac{p}{p-1}$. As $\mu_{1}$ is unimodal (with mode $m$), we get for $z$ positive
\begin{align}\label{eq2}
&\int\limits_{\bR}|f_{\mu}(x-z)-f_{\mu}(x)|\,\lambda(dx) \nonumber \\
=&\int\limits_{(-\infty,m)}f_{\mu}(x)-f_{\mu}(x-z)\,\lambda(dx)+\int\limits_{(m,m+z)}|f_{\mu}(x-z)-f_{\mu}(x)|\,\lambda(dx)\nonumber\\
&+\int\limits_{(m+z,\infty)}f_{\mu}(x-z)-f_{\mu}(x)\,\lambda(dx\nonumber)\\
=&\int\limits_{(-\infty,m)}f_{\mu}(x)\,\lambda(dx)+\int\limits_{(m+z,\infty)}f_{\mu}(x-z)\lambda(dx)\nonumber\\
&-\left(\int\limits_{(-\infty,m)}f_{\mu}(x-z)\,\lambda(dx) +\int\limits_{(m-z,m+z)}f_{\mu}(x)\,\lambda(dx)+\int\limits_{(m+z,\infty)}f_{\mu}(x)\,\lambda(dx) \right)\nonumber\\
&+\int\limits_{(m,m+z)}|f_{\mu}(x-z)-f_{\mu}(x)|\,\lambda(dx)+\int\limits_{(m-z,m+z)}f_{\mu}(x)\,\lambda(dx)\nonumber\\
=&1-1+\int\limits_{(m,m+z)}|f_{\mu}(x-z)-f_{\mu}(x)|\,\lambda(dx)+\int\limits_{(m-z,m+z)}f_{\mu}(x)\,\lambda(dx).
\end{align}
Now as $f\in L^{p^*}(\bR)$, we conclude that
\begin{align*}
&\int\limits_{\bR}|f_{\mu_1}(x-z)-f_{\mu_1}(x)|\,\lambda(dx)\\
\le&  \int\limits_{(m,m+z)}|f_{\mu_1}(x-z)|\lambda(dx)+ \int\limits_{(m,m+z)}|f_{\mu_1}(x)|\lambda(dx)+\int\limits_{(m-z,m+z)}f_{\mu_1}(x)\,\lambda(dx)\\
&\le ||f_{\mu_1}||_{L^{p^*}}z^{\frac{1}{p}}+ ||f_{\mu_1}||_{L^{p^*}}z^{\frac{1}{p}}+ 2^{\frac{1}{p}}||f_{\mu_1}||_{L^{p^*}}z^{\frac{1}{p}}
\le (2+{2^{\frac{1}p}})||f_{\mu_1}||_{L^{p^*}} \,z^{\frac{1}{p}}. 
\end{align*}
The assumption for $z<0$ follows by symmetry.\\
ii) Since $c=k(0+)+k(0-)>1$, it follows from [\ref{Sato}, Theorem 28.4] that $f_\mu$ is continuous on $\bR$. Hence we can bound the modulus by (\ref{eq2}) (for $z>0$) by
\begin{align*}
&\int\limits_{\bR}|f_{\mu}(x-z)-f_{\mu}(x)|\,\lambda(dx)\\
\le& \int\limits_{(m,m+z)}|f_{\mu}(x-z)-f_{\mu}(x)|\,\lambda(dx)+\int\limits_{(m-z,m+z)}f_{\mu}(x)\,\lambda(dx)\\
\le& \sup_{x\in\bR}|f_{\mu}(x)|\left(2\int\limits_{(m,m+z)}\lambda(dx)+\int\limits_{(m-z,m+z)}\,\lambda(dx)\right)\\
=&4\sup_{x\in\bR}|f_{\mu}(x)| z.
\end{align*}
Now we assume that $\mu$ is infinitely divisble with characteristic triplet $(0,\gamma,\nu)$ such that there exists $\delta>0$ such that $|x|\nu(dx)$ has a Lebesgue density $k$ in $(-\delta,\delta)$. We know that there exists for small $\varepsilon>0$ a $\rho>0$ such that $k(x)\ge \liminf_{x\to 0+}k(x)-\frac{\varepsilon}{2}>0$  for every $x\in (0,\rho)$ and $k(x)\ge \liminf_{x\to 0-}k(x)-\frac{\varepsilon}{2}>0$  for every $x\in (-\rho,0)$. So we can find a minorizing L\'{e}vy measure $l(x)/|x|\lambda(dx)$ for $\nu$ by setting $$l(x)=\one_{(0,\rho)}\left(\liminf_{x\to 0+}k(x)-\frac{\varepsilon}{2}\right)+\one_{(-\rho,0)} \left(\liminf_{x\to 0-}k(x)-\frac{\varepsilon}{2}\right)$$ 
Let $\mu_1$ be the self-decomposable distribution with triplet $(0,\gamma,\frac{l(x)}{|x|}\lambda(dx))$ and $\mu_2$ be the infinitely divisible distribution wtih triplet $(0,0,(\nu-\frac{l(x)}{|x|}\lambda)(dx))$. Then $\mu=\mu_1\ast\mu_2$ and since $\mu_1$ satisfies i) and ii) respectively, if $\varepsilon$ is chosen small enough, so does $\mu$ by Lemma \ref{lem1}.\\\\
iii) First assume that $\mu$ is such that $|x|\nu(dx)$ has a (bounded) Lebesgue density $k$ in $(-\delta,\delta)$ such that  $k$ is monotone on $(-\delta,0)$ and on $(0,\delta)$. Observe that for every $\varepsilon>0$ there exists a constant $C>0$ such that $|\hat{\mu}(z)|>C(1+|z|)^{-c-\varepsilon}$ for every $z\in\bR$, see [\ref{Trabs}, Proposition 1]. Moreover, we know by [\ref{Bray}, Corollary 3] that
\begin{align*}
\sup_{|x|\ge \frac{1}{|z|}}|\hat{\mu}(x)|\le C' \sup_{0\le h \le |z|}\int\limits_{\bR}|f_{\mu}(x-h)-f_{\mu}(x)|\,\lambda(dx)
\end{align*}
for some constant $C'$. So we see that
\begin{align*}
\tilde{C}|z|^{c+\varepsilon}\le C\left(1+\frac{1}{|z|}\right)^{-c-\varepsilon}\le C' \sup_{0\le h \le |z|}\int\limits_{\bR}|f_{\mu}(x-h)-f_{\mu}(x)|\,\lambda(dx)
\end{align*}
for some constants $C,\tilde{C}>0$ with $|z|<1$. Choosing $\varepsilon=\frac{1}p -c$ gives the claim in this special case.\\
For general $\mu$ we set $$l(x)=\one_{(0,\rho)}\left(\limsup_{x\to 0+}k(x)+\frac{\varepsilon}{2}\right)+\one_{(-\rho,0)} \left(\limsup_{x\to 0-}k(x)+\frac{\varepsilon}{2}\right)$$ and majorize $\nu$ by $l(x)/|x|\lambda|_{(-\delta,\delta)}(dx)+\nu|_{(-\delta,\delta)^c}(dx)$ which gives us our assertion by Lemma \ref{lem1}, as  otherwise the majorizing distribution would not satisfy iii).
\end{proof}
\begin{remark}\mbox{}\\
\begin{itemize} 
\item[i)]  For  Theorem \ref{theorem1} i) and ii) it is sufficient that $|x|\nu(dx)$ can be minorized by a measure with the sufficient conditions. Similarly, for Theorem \ref{theorem1} iii) it is sufficient that $|x|\nu(dx)$ can be majorized by a measure with the sufficient conditions and that $a=0$. This follows from Lemma \ref{lem1}.
\item[ii)] For Theorem \ref{theorem1} iii) one can give further conditions on $k$ such that $c_{\sup}=1/p$ is sufficient for (\ref{eq3}) to hold, see for example  [\ref{Trabs}, Proposition 1].
\end{itemize}
\end{remark}

Another example where we can apply the same techniques is a symmetric infinitely divisible distribution $\mu$ with characteristic triplet $(0,0,\nu)$ such that $\nu$ is unimodal and has mode $0$. Then also $\mu$ is unimodal with mode $0$, see [\ref{Sato}, Theorem 54.2]. 
\begin{corollary}\label{cor11}
Let $\mu$ be an infinitely divisible distribution with characteristic triplet $(0,0,\nu)$. Assume that
\begin{align*}
\liminf_{r\to 0} \frac{\int_{[-r,r]}x^2 \,\nu(dx)}{ r^2\log (\frac{1}{r})}=:C>\frac{1}{2p}
\end{align*}
for some $1<p\le 2$. Then $\mu$ has a Lebesgue density $f_\mu\in L^{1}(\bR,[0,\infty))\cap L^{p/(p-1)}(\bR,[0,\infty))$. Furthermore, if $\nu$ is unimodal with mode $0$ and $\mu$ is symmetric, then  there exists a constant $C>0$ such that
\begin{align*}
\int\limits_{\bR}|f_\mu(x-z)-f_\mu(x)|\,\lambda(dx)\le C |z|^{\frac{1}{p}}
\end{align*}
for every $z\in\bR$. \\
If additionally the condition 
\begin{align*}
\liminf\limits_{r\to 0}\frac{\int\limits_{[-r,r]}x^2\nu(dx)}{r^{2-\alpha}}>0,
\end{align*}
is satisfied for some $\alpha \in (0,2)$, we can bound the modulus by $|z|$ times a constant.
\end{corollary}
\begin{proof}
Assume that
\begin{align*}
\liminf_{r\to 0} \frac{\int_{[-r,r]}x^2 \,\nu(dx)}{ r^2\log (\frac{1}{r})}:=C>\frac{1}{2p}.
\end{align*}
 Then there exists a constant $\varepsilon>0$ such that $\int_{[-r,r]}x^2\nu(dx)\ge (C-\varepsilon) r^2\log \frac{1}{r}  $ for small enough $r$ and $C-\varepsilon>\frac{1}{2p}$. As $1-\cos(u)\ge2\left(\frac{u}{\pi}\right)^2$ for $|u|\le \pi$, we see that
\begin{align*}
|\hat{\mu}(z)|= &\exp\left( \int\limits_{\bR} (\cos(xz)-1) \nu(dx) \right)\\
\le &\exp\left( - \frac{2}{\pi^2}\int\limits_{|x|\le \pi/|z|} z^2x^2 \nu(dx) \right)\\
\le &\exp\left(-  \frac{2}{\pi^2}z^2(C-\varepsilon)\frac{\pi^2}{z^2}\log\left|\frac{z}{\pi}\right| \right)\\
=&\exp\left(-\log\left|\frac{z}{\pi}\right|^{2(C-\varepsilon)}\right)=\frac{\pi^{2(C-\varepsilon)}}{|z|^{2(C-\varepsilon)}}\le \frac{\pi^{2(C-\varepsilon)}}{|z|^{\frac{1}{p}+\delta}}
\end{align*} 
for some $\delta>0$ and $|z|$ great enough. It follows that $\hat{\mu}\in L^p(\bR,\bC)$ and from that we conclude that there exists a density, which is $p/(p-1)$-integrable Now if $\nu$ is additionally unimodal with mode $0$ then so is $\mu$, see [\ref{Sato}, Theorem 54.2]. By the same proof as in Theorem \ref{theorem1} i) we conclude that the modulus of continuity can be bounded by $|z|^{\frac{1}{p}}$ times a constant. If the L\'{e}vy-measure especially satisfies the condition
\begin{align*}
\liminf\limits_{r\to 0}\frac{\int\limits_{[-r,r]}x^2\nu(dx)}{r^{2-\alpha}}>0,
\end{align*}
the Lebesgue density is continuous, see [\ref{Sato}, Proposition 28.3] and the modulus of continuity is bounded by $|z|$ times a constant by the same proof as in Theorem \ref{theorem1} ii).\\
\end{proof}

\section{Densities of stochastic integrals}
In this section we look at distributions arising as stochastic integrals $\int_0^t g(s)dL(s)$ or $\int_0^\infty g(s)dL(s)$, when $g$ is a deterministic function and $L$ a L\'{e}vy process. A L\'{e}vy process is a real-valued stochastic process $L=(L_t)_{t\ge 0}$ with stationary and independent increments, such that $L_0=0$ almost surely and such that the paths of $L$ are right-continuous with finite left-limits. There exists a one-to-one correspondence between infinitely divisible distributions and L\'{e}vy processes (in law). In particular, the distribution of a L\'{e}vy process $L$ at time $1$ is infinitely divisible and characterizes the distribution of $L$. The characteristic triplet of $\mathcal{L}(L_1)$ is then also called the characteristic triplet of $L$. \\
The existence of the integrals $\int_0^t g(s)dL(s)$ or $\int_0^\infty g(s)dL(s)$ can be completely characterized by the characteristic triplet $(a,\gamma,\nu)$ of $L$ and $g$, see   [\ref{Rajput}, Theorem 2.7, p. 461]. Moreover, the integrals are infinitely divisible with characteristic triplet $(a_g,\gamma_g,\nu_g)$ where
\begin{align}
\gamma_g&=\int\limits_{[0,t)}\left( \gamma g(s)+\int\limits_{\bR}g(s)r(\one_{[-1,1]}(g(s)r)-\one_{[-1,1]}(r))\,\nu(dr)\right)\lambda(ds),\nonumber\\
a_g&=\int\limits_{[0,t)} a g(s)^2\,\lambda(ds)\textrm{ and}\label{eq4}\\
\nu_g(B)&=\int\limits_{[0,t)}\int\limits_{\bR} \one_{B\setminus\{0\}}(g(s)r)\,\nu(dr)\lambda(ds),\quad B\in\mathcal{B}\label{eq5}
\end{align}
with $t\in [0,\infty]$.
\subsection{Stochastic Integrals with compact support}
Now look at distributions of the form $Z=\int_{[0,t]}g(s)\,dL(s)$, where $t\in [0,\infty)$ and $L=(L_s)_{s\ge 0}$ is a L\'{e}vy process with characteristic triplet $(0,\gamma,\nu)$ with $\nu(\bR)>0$. We give sufficient conditions depending on $L$ and $g$ such that $Z$ satisfies the assumptions of Theorem \ref{theorem1}. We immediately restrict to the case when the Gaussian variance $a=0$, for otherwise $a_g>0$ by (\ref{eq4}) (unless $\int_0^t g(s)^2\,\lambda(ds)=0$) and hence Corollary \ref{cor10} can be applied. We start with the following lemma, where we write $\frac{x}{B}:=\{ \frac{x}{b}:b\in B\}$ for $x\in\bR$ and $B\subset \bR\setminus \{0\}$.
\begin{lemma}\label{cor2}
 Let $g:[0,t]\to \bR$ be a $\mathcal{C}^1$-Diffeomorphism onto its range.
\begin{itemize}
\item[i)] Then $|x|\nu_g(dx)$ is absolutely continuous with Lebesgue density $k$ given by
\begin{align*}
k(x)=\int_{\bR} \one_{g([0,t])}(x/r)\frac{|x|}{|r|}\left|(g^{-1})'(x/r)\right|\nu(dr)<\infty
\end{align*}
for all $x\in\bR\setminus \{0\}$. 
\item[ii)] Let $g>0$ in $[0,t]$. If $\liminf_{x\to 0+}\nu(\frac{x}{g([0,t])})=\lambda_1>0$, then $\liminf_{x\to 0+}k(x)\ge \inf_{y\in g([0,t])}|y| |(g^{-1})'(y)|\lambda_1$ and if $\limsup_{x\to 0+}\nu(\frac{x}{g([0,t])})=\lambda_2<\infty$, then \\$\limsup_{x\to 0+}k(x)\le  \sup_{y\in g([0,t])}|y||(g^{-1})'(y)|\lambda_2$.
\end{itemize}
\end{lemma}
\begin{proof}
i) We know from (\ref{eq5}) that for every $A\in\mathcal{B}(\bR)$
\begin{align*}
(|x|\nu_g)(A)=&\int\limits_{\bR}\int\limits_{[0,t]} |g(s)r|\one_{A}(g(s)r)\lambda(ds)\nu(dr)\\
=&\int\limits_{\bR}\int\limits_{r g([0,t])}\frac{|x|}{|r|}\left|(g^{-1})'(x/r)\right|\one_{A}(x)\lambda(dx)\nu(dr)\\
=&\int\limits_{\bR}\one_{A}(x)\int\limits_{\bR} \one_{g([0,t])}(x/r)\frac{|x|}{|r|} \left|(g^{-1})'(x/r)\right|\nu(dr)\lambda(dx).
\end{align*}
So we see that the density is given by $\int_{\bR} \one_{g([0,t])}(x/r)\frac{|x|}{|r|}\left|(g^{-1})'(x/r)\right|\nu(dr)$. Observe that the integral is taken for every $x \neq 0$ in a set away from zero, so boundedness is enough for the finiteness of the integral.\\
ii) Now assume that $g>0$. We see that for $x\in\bR\setminus\{0\}$
\begin{align*}
&\int_{\bR} \one_{g([0,t])}(x/r)\frac{|x|}{|r|}\left|(g^{-1})'(x/r)\right|\nu(dr)\\
\le & \sup_{y\in g([0,t])}|y||(g^{-1})'(y)|\int_{x/g([0,t])}\nu(dr)\\
=& \sup_{y\in g([0,t])}|y||(g^{-1})'(y)| (\nu(x/g([0,t]))
\end{align*} 
and
\begin{align*}
\int_{\bR} \one_{g([0,t])}(x/r)\frac{|x|}{|r|}\left|(g^{-1})'(x/r)\right|\nu(dr)\ge \inf_{y\in g([0,t])}|y||(g^{-1})'(y)|\nu\left(\frac{x}{g([0,t])}\right).
\end{align*}
The rest follows by taking the limits.\\
\end{proof}
\begin{remark}
For the existence of a Lebesgue density of $\nu_g$ it is enough to assume that preimages of Lebesgue null sets under $g$ are again Lebesgue null sets, a condtion called Luisin $(N^{-1})$-condition. To see this, let $B\in\mathcal{B}$ be a Lebesgue null set. Then so is $\frac{1}{r}(B\setminus\{0\})$ for every $r\neq 0$ and hence by $(\ref{eq5})$ and the Lusin $(N^{-1})$-condition we obtain
\begin{align*}
\nu_g(B)&=\int\limits_{\bR} \int\limits_{[0,t]} \one_{g^{-1}\left(\frac{1}{r}(B\setminus \{0\})\right)}(s)\lambda(ds)\nu(dr)\\
&=\int\limits_{\bR} \lambda\left(g^{-1}\left(\frac{1}{r}\left(B\setminus\{0\}\right)\right)\right)\nu(dr)=0.
\end{align*}
This shows that $\nu_g$ is absolutely continuous and hence has a density. Sufficient conditions for the Lusin $(N)^{-1}$-conditions to hold can be found in
 [\ref{Hencl}, Theorem 4.13, p. 74].
\end{remark}
As a consequence of Lemma \ref{cor2} and Theorem \ref{theorem1} we find sufficient conditions for the existence of a Lebesgue density of bounded variation.
\begin{corollary}\label{cor4}
Let $g:[0,t]\to \bR$ be a $\mathcal{C}^1-$diffeomorpism onto its range and $L$ be a L\'{e}vy process with characteristic triplet $(0,\gamma,\nu)$ with $\nu(\bR)>0$. Let $Z=\int_{[0,t]}g(t)\,dL(t)$.\\
i) Let $\nu(\bR)=\infty$. Then the distribution of $Z$ is absolutely continuous.\\
ii)  Let $g>0$ on $[0,t]$. If $$\left(\liminf_{x\to 0+}\nu\left(\frac{x}{g([0,t])}\right)+\liminf_{x\to 0-}\nu\left(\frac{x}{g([0,t])}\right)\right)\inf_{y\in g([0,t])}|y||(g^{-1})'(y)|>1,$$ then $Z$ has a density which is of bounded variation.\\
iii)  Let $g>0$ on $[0,t]$.  If  $$\left(\limsup_{x\to 0+}\nu\left(\frac{x}{g([0,t])}\right)+\limsup_{x\to 0-}\nu\left(\frac{x}{g([0,t])}\right)\right)\sup_{y\in g([0,t])}|y||(g^{-1})'(y)|<1,$$ then the density of the random variable $Z$ (if existent) cannot be of bounded variation. 
\end{corollary}
\begin{proof}
i) This follows by $(\ref{eq5})$, Lemma \ref{cor2} i) and [\ref{Sato}, Theorem 27.7, p. 177].\\
ii) + iii) Clear by Theorem \ref{theorem1} and Lemma \ref{cor2} ii). Observe the condition in ii) implies $\nu(\bR)=\infty$ such that $Z$ has a density by i).
\end{proof}
\begin{example}
Let us look at the L\'{e}vy-measure $\nu(dx)=\sum\limits_{n=0}^\infty k_{n}\delta_{b^{-n}}(dx)$ for some integer $b\in \bN\setminus\{1\}$ such that $\sum\limits_{n=0}^\infty k_n=\infty$ and $\sup_{n\in\bN}k_n\le C <\infty$ for some positive $C>0$. It is indeed a L\'{e}vy measure as even $$\int_{\bR}\min\{1,x\}\nu(dx)=\sum\limits_{n=0}^\infty k_nb^{-n}\le C\sum\limits_{n=0}^\infty b^{-n}=\frac{C}{1-b^{-1}}<\infty.$$  It is known that the one-dimensional distribution of the L\'{e}vy process $L$ with characteristic triplet $(0,0,\nu)$ is continuously singular, see [\ref{Sato}, Theorem 27.19]. Let $g:[0,1]\to\bR$ be a positive, increasing $\mathcal{C}^1$ diffeomorphism onto its range with $\frac{g(1)}{g(0)}\ge b^l$ for some $l\in\bN$. Let $x\in[0,1]$. We know that there exists an $n\in\bN$ such that $b^{-n}<x\le b^{-n+1}$. We have that $$\nu\left(\left[x,\frac{g(1)}{g(0)}x\right] \right)\ge \nu\left((b^{-n},b^{-n+1+l}]\right)=\sum\limits_{r=n-l-1}^{n-1}k_{r}$$ and see by Corollary \ref{cor4} ii) that if there exist $\varepsilon>0$ and $m\in\bN$, $m\ge l+1$, such that $\sum\limits_{r=i-l-1}^{i-1}k_{r}\ge \frac{1+\varepsilon}{\inf_{y\in g([0,1])}|y||(g^{-1})'(y)}$ for every $i\ge m$ then the density of the random variable $Z=\int\limits_{0}^1 g(t) dL(t)$ is of bounded variation (observe that $\liminf_{x\to 0+}\nu\left(\frac{x}{g([0,t])}\right)=\liminf_{x\to 0+}\nu\left(\left[x,\frac{g(1)}{g(0)}x\right]\right)$). Examples of such sequences $(k_n)_{n\in\bN}$ are easily constructed.\\
Now let $g$ be an increasing positive $\mathcal{C}^1$-diffeomorphism onto its range with $\frac{g(1)}{g(0)}\le b^l$ for some $l\in\bN$. Then we have $\nu\left( [x,\frac{g(1)}{g(0)}x] \right)\le \nu\left([b^{-n+1},b^{-n+1+l}]\right)=\sum\limits_{r=n-l-1}^{n-1}k_{r}$ and we see that if there exist $\varepsilon>0$ and an $m\in\bN$, $m\ge l+1$, such that  $\sum\limits_{r=i-l-1}^{i-1}k_{r}\le \frac{1-\varepsilon}{\sup_{y\in g([0,1])}|y||(g^{-1})'(y)}$ for every $i>m$, then by Corollary \ref{cor4} iii) the density of $Z$ is not of bounded variation (the density exists by Corollary \ref{cor4} i)). It is easy to construct such examples. Observe that they satisfy $\nu(\bR)=\infty$, hence positive $\mathcal{C}^1$-diffeomorhisms and $\nu(\bR)=+\infty$ do not imply bounded variation of the density of $Z$.\\
\end{example}
\begin{example}
Let $\nu(dx)=\sum\limits_{n=0}^\infty k_n \delta_{b^{-2^{n}}}$ with $b>1$, $\sum_{n=0}^\infty k_n=\infty$ and $\sup_{n\in\bN}k_n\le C<\infty$. Let $g:[0,1]\to \bR^+$ be a positive increasing $\mathcal{C}^1$-diffeormorphism onto its range and $m>0$ such that $\frac{g(1)}{g(0)}=b^m$. Then it is relatively easy to see that $$\nu\left(\left[\frac{x}{g(1)},\frac{x}{g(0)}\right]\right)=k_{n-1}$$ if there exists an $n\in\bN$ such that $x\in [g(1)b^{-2^{(n-1)}-m},g(1)b^{-2^{(n-1)}}]$, otherwise the term is equal to $0$ for $x$ small enough.  We see directly that we cannot use Corollary \ref{cor4} ii) anymore to give a sufficient condition for the density to be of bounded variation since $\liminf_{x\to 0+}\nu\left(\frac{x}{g([0,t])}\right)=0$, but if $k_n\le  \frac{1-\varepsilon}{\sup_{y\in g([0,1])}|y||(g^{-1})'(y)}$ for some $\varepsilon>0$ for every $n>n_0\in\bN$ then the density is not of bounded variation. 
\end{example}
Now assume that we have a non-deterministic L\'{e}vy-process $L=(L_{t})_{t \ge 0}$, $\mathcal{L}(L_1)$ being self-decomposable, with characteristic triplet $(0,\gamma, \nu)$, with $\frac{l(x)}{|x|}$ the Lebesgue-density of the L\'{e}vy-measure and a bounded strictly positive function $g>0$ on an interval $[0,t]$. This is as in Corollary \ref{cor4}, but observe that we no longer assume that $g$ is a $\mathcal{C}^1$-diffeomorphism on the cost of more restrictive conditions on $L$. It follows from (\ref{eq5}) that the L\'{e}vy measure of $Z$ and hence also $Z$ has a density $f_g$ by [\ref{Sato}, Theorem 27.7]. 
\begin{corollary}\label{cor3}
 Let $Z$ be as above with density $f_g\in L^1(\bR,[0,\infty))$. 
\begin{itemize}
\item[i)] If $l(0+)+l(0-)>1/(pt)$ with $p\in (1,2]$, then there exists a constant $C>0$ such that
\begin{align*}
\int\limits_{\bR}|f_g(x-z)-f_g(x)|\,\lambda(dx)\le C |z|^{\frac{1}p}
\end{align*}
for every $z\in\bR$.
\item[ii)] If $l(0+)+l(0-)>1/t$, then there exists a constant $C>0$ such that
\begin{align*}
\int\limits_{\bR}|f_g(x-z)-f_g(x)|\,\lambda(dx)\le C |z|
\end{align*}
for every $z\in\bR$.
\item[iii)] If $l(0+)+l(0-)<\frac{1}{pt}$  with $p\in (0,\infty)$ and $a=0$, then
\begin{align*}
\sup_{0\le h \le |z|}\int\limits_{\bR}|f_{g}(x-h)-f_{g}(x)|\,\lambda(dx)\ge C |z|^{\frac{1}{p}}
\end{align*}
for some constant $C>0$ and $z\in(-1,1)$.
\end{itemize}
\end{corollary}
\begin{proof}
The characteristic triplet of $\widehat{\mathcal{L}(Z)}$ is given by $(0,\gamma_g,\nu_g)$ as before, where
\begin{align*}
\nu_g(B)&=\int\limits_{[0,t]}\int\limits_{\bR} \one_{B}(g(s)r)\frac{l(r)}{|r|}\lambda(dr)\lambda(ds)
\end{align*}
by (\ref{eq5}). By easy calculations we find that 
\begin{align}\label{eq6}
k(r)/|r|:=\int_{[0,t]}l(r/g(s))\lambda(ds)/|r|
\end{align}
 is the Lebesgue density of $\nu_g$. Then
\begin{align*}
c_g:&=k(0+)+k(0-)\\
&=\lim_{r\to 0+}\int_{[0,t]}l(r/g(s))\lambda(ds)+\lim_{r\to 0-}\int_{[0,t]}l(r/g(s))g(s)\lambda(ds)\\
&=\int_{[0,t]}l(0+)\lambda(ds)+\int_{[0,t]}l(0-)\lambda(ds)\\
&=(l(0+)+l(0-))t,
\end{align*}
and the assertions follow by Theorem \ref{theorem1}.
\end{proof}
\begin{remark}
It follows from (\ref{eq6}) that in the situation of Corollary \ref{cor3} the distribution of $Z$ is also self-decomposable. By Corollary \ref{cor3} iii) we see that its probability density is not of bounded variation if the L\'{e}vy measure $\frac{l(x)}{|x|}\lambda(dx)$ satisfies $l(0+)+l(0-)<1/t$. As this property is independent of $g$, we see that for fixed $t$ we cannot find a positive $\mathcal{C}^1$-diffeomorphism for every characteristic triplet such that the stochastic integral has a density of bounded variation.\\
\end{remark}
\subsection{Stochastic integrals with non-compact supports}
Now we want to prove some aspects of the densities of distributions of the form $\int_{[0,\infty)}g(t)dL(t)$, whenever such an integral exists. As before we assume that $L$ has characteristic triplet $(0,\gamma,\nu)$ with $\nu(\bR)>0$. We assume that $g$ is a strictly positive, continuous function which attains its maximum $$c:=\max_{t\in[0,\infty)}g(t)$$ and that there exists a decomposition $(t_i)_{i\in\bN_0}$ with $0=t_0<t_1<\dotso$ and $t_i\to\infty$ for $i\to\infty$ such that $g$ restricted to $(t_i,t_{i+1})$ is a $\mathcal{C}^1-$diffeomorphism onto its range for every $i\in\bN_0$. Then we can write $$\int_{[0,\infty)}g(t)\,dL(t)=\sum_{i=0}^\infty \int_{t_i}^{t_{i+1}}g(t)\,dL(t)$$
where the limit is taken in probability and from Lemma \ref{cor2} i) we see that $\int_{[0,\infty)}g(t)dL(t)$ has a L\'{e}vy $\nu_g$ measure with Lebesgue density 
\begin{align*}
\frac{k(x)}{|x|}:=\frac{1}{|x|} \sum\limits_{i\in\bN_0} \int_{\bR}\one_{g((t_i,t_{i+1}))}(x/r)\frac{|x|}{|r|}|(g^{-1})'(x/r)|\nu(dr).
\end{align*}
From $(\ref{eq5})$ we further see that $\nu_g(\bR)=+\infty$, so that $\int_{[0,\infty)}g(t)dL(t)$ has a Lebesgue density by [\ref{Sato}, Theorem 27.7, p. 177].
Now we can write the density of the L\'{e}vy measure for $x>0$ as
\begin{align}
\frac{k(x)}{|x|}=&\frac{1}{|x|}\int\limits_{\bR} \sum\limits_{i\in\bN_0}\one_{g((t_i,t_{i+1}))}(x/r)\frac{|x|}{|r|}|(g^{-1})'(x/r)|\nu(dr)\nonumber\\
=& \frac{1}{|x|}\int_{\frac{x}{g((0,\infty))}} h(x/r)\nu(dr)\nonumber\\
=&\frac{1}{|x|}\int\limits_{\left[\frac{x}{c},\infty\right)} h(x/r)\nu(dr)\qquad a.e.,\label{eq6}
\end{align}
with $$h(s):=\sum_{i\in I_s} |s| |(g^{-1}|_{(t_i,t_{i+1})})'(s)|,$$where $I_s=\{i\in\bN_0: s\in g((t_{i},t_{i+1}))\}$. Similarly, $\frac{k(x)}{|x|}=\frac{1}{|x|}\int_{(-\infty,\frac{x}{c}]} h\left(\frac{x}{r}\right)\nu(dr)$ for $x<0$. Now we obtain immediately by Theorem \ref{theorem1}:
\begin{proposition}\label{lemma10}
Let $g$ have the same properties as above. \\
i) The random variable $\int_{[0,\infty)}g(t)\,dL(t)$, if existent, has a density of bounded variation, if $$\liminf_{x\to 0+}\int_{\left[\frac{x}{c},\infty\right)} h(x/r)\nu(dr)+\liminf_{x\to 0-}\int_{\left(-\infty,\frac{x}{c}\right]} h(x/r)\nu(dr)> 1.$$\\
ii) The random variable $\int_{[0,\infty)}g(t)\,dL(t)$, if existent, has not a density of bounded variation, if $$\limsup_{x\to 0+}\int_{\left[\frac{x}{c},\infty\right)} h(x/r)\nu(dr)+\limsup_{x\to 0-}\int_{\left(-\infty,\frac{x}{c}\right]} h(x/r)\nu(dr)< 1.$$\\
\end{proposition}
 If the integral is existent, it is known that there exists a sequence $(z_n)_{n\in\bN}$ such that $z_n\to \infty$ and $g(z_n)\to 0$ for $n\to\infty$. We use this simple fact to prove our next corollary.
\begin{corollary}\label{prop10}
Let $g:[0,\infty)\to (0,\infty)$ have the same properties as above, denote $T:=\{t_i:i\in\bN\}$ and assume that $$\liminf_{x\to\infty, x\notin T}\left|\frac{g(x)}{g'(x)}\right|=\alpha$$ for some $\alpha\in(0,\infty]$. Then $\int_{[0,\infty)}g(t)\,dL(t)$ has a density of bounded variation, if $\nu(\bR)> \frac{1}{\alpha}$.
\end{corollary}
\begin{proof}   Assume that $\liminf_{x\to\infty,x\notin T}\left|\frac{g(x)}{g'(x)}\right|=\alpha$ for some $\alpha\in(0,\infty]$. \\ We define the function $\tilde{h}:(0,c]\to \bR^+\cup \{\infty\}$ by $\tilde{h}(x)=h(x)$ for all $x\in (0,c]\setminus g(T)$ and $\tilde{h}(x)=\infty$ otherwise. Then it holds for $x>0$ that
\begin{align*}
k(x)=\int_{\left[\frac{x}{c},\infty\right)}h\left(\frac{x}{r}\right)\nu(dr)=\int_{\left[\frac{x}{c},\infty\right)\setminus \{\frac{x}{r}\in g(T)\}}h\left(\frac{x}{r}\right)\nu(dr)+\int_{ \{\frac{x}{r}\in g(T)\}}h\left(\frac{x}{r}\right)\nu(dr).
\end{align*}
Now as $\nu$ has a countable number of points with positive mass we conclude that only in the set $\{\frac{x}{r}\in g(T)\}\cap \{r\in B\}$, where $B$ is the set of points with a positive mass of $\nu$, $\int_{ \{\frac{x}{r}\in g(T)\}}h\left(\frac{x}{r}\right)\nu(dr)$ is unequal to $0$. So we see that we only differ on a  Lebesgue null set by considering $\tilde{k}(x)=\int_{\left[\frac{x}{c},\infty\right)}\tilde{h}\left(\frac{x}{r}\right)\nu(dr)$ instead of $k$. Oberserve that the same arguments work for $x<0$.\\
Let $x_n\to 0+$ for $n\to\infty$ with $x_n\notin g(T)$ and choose $y_n\to\infty$ with $g(y_n)=x_n$ (existent since $g$ is continuous and by the observation above). We see that
\begin{align*}
\liminf_{n\to\infty} \tilde{h}(x_n)\ge\liminf_{n\to\infty}\frac{|x_n|}{|g'(y_n)|}=\liminf_{n\to\infty}\frac{|g(y_n)|}{|g'(y_n))|}\ge \alpha,
\end{align*}
as $y_n\to\infty$. Therefore we obtain by the Lemma of Fatou
\begin{align*}
\liminf_{x\to 0+} \tilde{k}(x)\ge \alpha \nu((0,\infty)) \quad\textrm{ and }\quad \liminf_{x\to 0-} \tilde{k}(x)\ge \alpha \nu((-\infty,0)).
\end{align*}
Proposition \ref{lemma10} implies  that $\int_{[0,\infty)}g(t)dL(t)$ has a density of bounded variation if $\nu(\bR) > \frac{1}{\alpha}$.
\end{proof}
\begin{remark}\label{rem1} We could also use other specifications for $g$. For example consider a strictly positive and continuous function $g$ on $[0,\infty)$ such that there exist sequences $(a_n)_{n\in\bN}$ and $(b_n)_{n\in\bN}$ with $0<a_n<b_n\le a_{n+1}$ for every $n\in\bN$ such that $g|_{(a_n,b_n)}$ is a $\mathcal{C}^1$-diffeomorphism onto its range and $g(\cup_{n=m}^\infty [a_n,b_n))$ is a half-open interval with a maximum $c<\infty$ and infimum $0$ for an $m\in\bN$, i.e. $g(\cup_{n=m}^\infty [a_n,b_n))=(0,c]$. For these kind of functions Proposition \ref{lemma10} i) and Corollary \ref{prop10} also hold true, where $$h(s):=\sum_{i\in I_s} |s| |(g^{-1}|_{(a_i,b_{i})})'(s)|,$$ with $I_s=\{i\in\bN_0: s\in g((a_{i},b_{i}))\}$.
\end{remark}
\begin{example}
 Applying Corollary \ref{prop10} to the function $g(x)=e^{-bx}$ with $b>0$ gives $\alpha=\frac{1}{b}$, hence $\int_{[0,\infty)}e^{-bt}dL(t)$ has a density of  bounded variation if $\nu(\bR)>\frac{1}{b}$. Applying Corollary \ref{prop10} (more precisely, the extension according to Remark \ref{rem1}) to the function $g(x)=\min\{x^{-p},C\}$ with $p>0$ gives $\alpha=\infty$. Hence $\int_{[0,\infty)}\min\{t^{-p},C\} dL(t)$ has a density of bounded variation when $\nu(\bR)>0$.
\end{example}
If $g(x)=e^{-x^2}$ we cannot use Corollary \ref{prop10} as $\frac{g(x)}{g'(x)}=1/(2x)\to 0$ for $x\to\infty$. We  will give another condition such that we can obtain sufficient conditions for the existence of a probability density of bounded variation implied by such a kernel function.
\begin{corollary}
Let $g(x)=e^{-\psi(x)}$ with $\psi:[0,\infty)\to \bR$ continuous such that $\psi:(0,\infty)\to(0,\infty)$ is a strictly increasing $\mathcal{C}^1$-diffeomorphism and such that $\psi(0)=0$ and $(\psi^{-1})'$ is decreasing. Then the Lebesgue density of $\int_{[0,\infty)}g(t)dL(t)$ is of bounded variation if 
\begin{align*}
&\liminf_{x\to 0+}(\psi^{-1})'(-\log(x))\nu((x,1))+\liminf_{x\to 0-}(\psi^{-1})'(-\log|x|)\nu((-1,x))\\
=&\liminf_{x\to 0+}\frac{\nu((x,1))}{\psi'(\psi^{-1}(-\log(x)))}+\liminf_{x\to 0-}\frac{\nu((-1,x))}{\psi'(\psi^{-1}(-\log|x|))}>1.
\end{align*}
\end{corollary}
\begin{proof}
For simplicity of notation, we assume that $\nu((-\infty,0))=0$. A direct calculation gives us from (\ref{eq6}) that
\begin{align*}
k(x)=\int_{(x,\infty)}(\psi^{-1})'(\log(r)-\log(x))\nu(dr)=\int_{(x,\infty)}\frac{1}{\psi'(\psi^{-1}(\log(r)-\log(x)))}\nu(dr).
\end{align*}
As $(\psi^{-1})'$ is decreasing we see that for $0<x<1$
\begin{align*}
k(x)\ge \int_{(x,1)} (\psi^{-1})'(-\log(x))\nu(dr)= (\psi^{-1})'(-\log(x))\nu((x,1)).
\end{align*}
So we see by Proposition \ref{lemma10} that if $$\liminf_{x\to 0+}(\psi^{-1})'(-\log(x))\nu((x,1))=\liminf_{x\to 0+}\frac{\nu((x,1))}{\psi'(\psi^{-1}(-\log(x)))}>1$$ the Lebesgue density of $\int_{[0,\infty)}g(t)dL(t)$ is of bounded variation.
\end{proof}
\begin{example}
Let $\psi(x)=x^p$ for $p>1$. Then we have $\psi^{-1}(x)=x^{1/p},\, (\psi^{-1})'(x)=\frac{1}{p}x^{1/p-1}$, which is decreasing. We see that if $$\liminf_{x\to 0+}\frac{\nu((x,1))}{\left(\log\left(\frac{1}{x}\right)\right)^{1/p-1}}+\liminf_{x\to 0-}\frac{\nu((-1,x))}{\left(\log\left(\frac{1}{|x|}\right)\right)^{1/p-1}}>p$$
the Lebesgue density of $\int_{[0,\infty)}e^{-t^p}dL(t)$ is of bounded variation.
\end{example}
\section*{Acknowledgement:} The author would like to thank Alexander Lindner for his patience and support and for many interesting and fruitful discussions.

David Berger\\
 Ulm University, Institute of Mathematical Finance, Helmholtzstra{\ss}e 18, 89081 Ulm,
Germany\\
email: david.berger@uni-ulm.de
\end{document}